\documentclass[12pt]{amsart}

\usepackage{xcolor} 
\usepackage{graphicx} 
\usepackage{enumerate}

\usepackage{amsmath}
\usepackage{amssymb}
\usepackage{amsrefs}

\theoremstyle{plain}
\newtheorem{theorem}{Theorem}[section]
\newtheorem{lemma}[theorem]{Lemma}
\newtheorem{proposition}[theorem]{Proposition}
\newtheorem{corollary}[theorem]{Corollary}
\newtheorem{definition}[theorem]{Definition}

\theoremstyle{definition}

\newtheorem{question}{Question}[section]

\newtheorem{claim}{Claim}

\theoremstyle{remark}

\numberwithin{equation}{section}

\DeclareMathOperator{\cf}{cf}
\DeclareMathOperator{\dom}{dom}

\DeclareMathOperator{\supp}{supp}

\DeclareMathOperator{\next}{\mathop{next}}

\DeclareMathOperator{\val}{\mathop{val}}

\def\trestriction{{\restriction}}

\newcommand{\forces}[2]{\Vdash_{#1} \mbox{``\,} #2 \mbox{''}}

\newcommand{\tie}[1]
{\,\raise-5pt\hbox{
${\buildrel{\displaystyle{\rhd}\!\!{\lhd}}\over{
\scriptstyle
#1}}$}\,}

\title{Asymmetric tie-points and almost clopen subsets of $\mathbb {N}^*$
}

\author[A. Dow]{Alan Dow}
\address{Department of Mathematics,
University of North Carolina at Charlotte, 
Charlotte, NC 28223}
\email{adow@uncc.edu}
\author[S. Shelah]{Saharon Shelah}
\address{Department of Mathematics, Rutgers University, Hill Center,
 Piscataway, 
 New Jersey, U.S.A. 08854-8019}
\curraddr{Institute of Mathematics\\Hebrew University\\
Givat Ram, Jerusalem 91904, Israel}
\email{shelah@math.rutgers.edu}

\date{\today}

\thanks{The research of the first author was supported by
 the NSF grant No. NSF-DMS 1501506.
The research of the second  author was
 supported by
   the United States-Israel Binational Science Foundation 
(BSF Grant
   no. 2010405), and by the 
 NSF grants No. NSF-DMS 1101597  and 136974.}

\keywords{
ultrafilters, cardinal invariants of continuum
}

\subjclass{54D80, 03E15 }

\begin{document}
\begin{abstract}
  A tie-point  of a compact space is analogous to a cut-point:
the complement of the point falls apart into two relatively clopen
non-compact subsets. Set-theoretically a tie-point of $\mathbb N^*$
is an
ultrafilter whose dual maximal  ideal can be generated by
the union of two non-principal  mod finite orthogonal ideals.
We review some of the many consistency results
that have depended on the construction of tie-points of $\mathbb
N^*$. One especially important application, due to Velickovic,
was to the existence of non-trivial involutions on $\mathbb N^*$. 
A tie-point of $\mathbb N^*$ has been called symmetric if
it is the unique fixed point of an involution. 
We define the notion of an almost clopen set to be the closure of
one of the proper relatively clopen  subsets of the complement of a
tie-point.  We explore asymmetries of almost clopen subsets of
$\mathbb N^*$ in the sense of how may an almost clopen set differ
from its natural complementary almost clopen set.
\end{abstract}
\maketitle

\bibliographystyle{plain}

\section{Introduction}

In this introductory section we review some background to motivate our
interest in further 
study  of tie-points and almost clopen sets.
The Stone-{\v C}ech compactification of the integers $\mathbb N$, is
denoted as $\beta\mathbb N$ and, as a set, is equal to $\mathbb N$
together with all the free ultrafilters on $\mathbb N$. The remainder
 $\mathbb N^* = \beta\mathbb N\setminus \mathbb N$ can be
 topologized as a
 subspace of the Stone space of the power set  of $\mathbb N$ as a
 Boolean algebra and, in particular, for a subset $a$ of $\mathbb N$,
 the set $a^* $ of all free ultrafilters with $a$ as an element, is
 a basic clopen subset of $\mathbb N^*$.  Set-theoretically it is
 sometimes more 
 convenient to work with the set of ordinals $\omega$ in place of the
 natural numbers $\mathbb N$, and the definitions of $\omega^*$ and
  $a^*$ for $a\subset \omega$ are analogous.

 A point $x$ of a space $X$ is a butterfly point (or $b$-point
  \cite{Sapirbpoint}) if there are sets $D,E\subset X\setminus \{x\}$
  such that $\{x\} = \overline{D}\cap \overline{E}$. In
   \cite{DSh1}, the authors introduced the tie-point
terminology. 

 \begin{definition} 
A point $x$ is a tie-point of a space $X$ if there are closed sets
$A,B$ of $X$ such that $X=A\cup B$, $\{x\}=A\cap B$ and $x$ is a 
limit  point of each of $A$ and $B$. 
 We picture (and denote) this as  $X = A\tie{x} B$ where $A,
B$ are the closed sets which have a unique common accumulation point
$x$ and
say that $x$ is a tie-point as witnessed by $A,B$.
 \end{definition}

In this note the focus is on the local properties of $x$ with respect
to each of the closed sets $A$ and $B$ such that
$A\tie{x}B$ in the case when $A,B$ witness that $x$ is a tie-point.
For this reason we
introduce the notion of an almost clopen subset of $\mathbb N^*$.

\begin{definition}
A set $A\subset \mathbb N^*$ is almost clopen if $A$ is the closure of
an open subset of $\mathbb N^*$ and has a unique boundary point, which
we denote $x_A$. 
\end{definition}

\begin{proposition}
If $A$ is an almost clopen subset of $\mathbb N^*$, then $B =
\{x_A\}\cup (\mathbb N^*\setminus A)$ is almost clopen and
$x_B = x_A$. In addition $x_A$ is a tie-point as witnessed by
 $A,B$.
\end{proposition}

\begin{definition}\cite{DSh1}
  A tie-point  $x$ is a symmetric
tie-point of $\mathbb N^*$ if there is a pair
$A,B$ witnessing that $x$ is a tie-point and if
there is a homeomorphism $h: A \rightarrow  B$  satisfying
 that $h(x)=x$. 
\end{definition}

If $A$ is almost clopen, then we refer to $B = \{x_A\}\cup (\mathbb
N^* \setminus A)$ as the almost clopen complement of  $A$. 
A more set-theoretically inclined reader would surely prefer a
staightforward translation of almost clopen to properties of ideals 
of subsets of $\mathbb N$ and the usual 
mod finite ordering $\subset^*$. 

\begin{definition} If $A$ is any  subset of $\mathbb N^*$,
 then $\mathcal I_A$ is defined as the set
 $\{ a\subset \mathbb N : a^* \subset A\}$.
\end{definition}

For any family  $\mathcal A$ of subsets of $\mathbb N$ (or $\omega$),
we define $\mathcal A^\perp$ to be the orthogonal ideal
$\{ b\subset \mathbb N :
(\forall a\in  \mathcal A)\  
b\cap a=^*\emptyset\}$. Let us note that if $\mathcal I$
is an ideal that
 has no $\subset^*$-maximal element,
 then the ideal generated by
  $\mathcal I\cup \mathcal I^\perp$ is a proper ideal.

\begin{lemma}
If $A$ is an almost clopen subset of $\mathbb N^*$ 
with almost clopen complement $B$,
 then 
$\mathcal I_A\cap \mathcal I_B $ is the Frechet ideal 
  $[\mathbb N]^{<\aleph_0}$,
 $\mathcal I_B = \mathcal I_A^\perp$,
 and  $x_A$ is the unique ultrafilter
that is
 disjoint from $\mathcal I_A\cup \mathcal I_B$.
\end{lemma}

Almost clopen sets (and tie-points) first arose implicitly in the work
of Fine and Gillman \cite{FG} in the investigation of extending
continuous functions on dense subsets of $\mathbb N^*$. A subset
$Y$  
of a space $X$ is $C^*$-embedded if every bounded continuous
real-valued function on $Y$ can be continuously extended to all of
$X$. The character of a point $x\in \mathbb N^*$ is the minimal
cardinality of a filter base for $x$ as an ultrafilter on $\mathbb
N$. 

\begin{proposition}(\cite{FG})
If $x$ is a tie-point of $\mathbb N^*$, then $\mathbb N^*\setminus
\{x\}$ is not $C^*$-embedded in $\mathbb N^*$. Every point of
character $\aleph_1$ is a tie-point of $\mathbb N^*$.
\end{proposition}

It was shown in \cite{vDKvM}  to be consistent with ZFC that $\mathbb
N^*\setminus \{x\}$ is  $C^*$-embedded for all $x\in \mathbb N^*$.
It was also shown by Baumgartner \cite{BaPFA} that their result holds
in models of the Proper Forcing Axiom (PFA).

\begin{proposition}(\cites{vDKvM, BaPFA})
  The proper forcing axiom implies $\mathbb N^*\setminus \{x\}$ is
  $C^*$-embedded in $\mathbb N^*$ for all $x\in \mathbb N^*$
\end{proposition}

\begin{corollary} PFA implies
that  there are no almost clopen sets and no tie-points in
  $\mathbb N^*$.
\end{corollary}

 Almost clopen sets arise in the study of minimal  extensions of Boolean
 algebras (\cite{Kopp}) and in the application of this method of
 construction for building a variety of counterexamples
  (e.g. \cites{KoszTAMS,JKS,Rabus, DSh3}). The next application of
  almost clopen subsets of $\mathbb N^*$ were to the study of
  non-trivial automorphisms of $\mathcal P(\mathbb N)/fin$, or
  non-trivial autohomeomorphisms of $\mathbb N^*$. Kat\v etov
  \cite{Katetov}  proved
  that the set of fixed points of an autohomeomorphism 
  of $\beta\mathbb N$  will be a clopen set. It is immediate
  from Fine and Gillman's work
  in
  \cite{FG} that every $P$-point of character of $\aleph_1$
  is a fixed point of a non-trivial autohomeomorphism of
   $\mathbb N^*$.

\begin{definition}
  A point $x$ of $\mathbb N^*$ is a $P$-point if
  the ultrafilter $x$  is countably complete mod finite.
  For a cardinal $\kappa$,
an ultrafilter
$x$ on $\mathbb N$ is a simple $P_\kappa$-point if
$x$ has a base well-ordered by mod finite inclusion
of order type $\kappa$. 
\end{definition}

\begin{proposition}\cite{FG}
  If $A$ is an almost clopen subset of $\mathbb N^*$
  and 
  $x_A$ is a simple $P_{\aleph_1}$-point of $\mathbb N^*$,
  then
  \begin{enumerate}
  \item  $A$ is homeomorphic to $\mathbb N^*$,
  \item  $x_A$ is a symmetric tie-point,
  \item there is an autohomeomorphism $f$ on $N^*$ such
    that $\{x\}$ is the only fixed point of $f$.
\end{enumerate}
\end{proposition}

As we have seen above, PFA implies that there are no almost clopen
subsets of $\mathbb N^*$, and of course, PFA also implies that
all autohomeomorphisms of $\mathbb N^*$ are trivial
\cite{ShStPFA}.  
However 
Velickovic utilized the simple $P$-point trick
 (motivating our definition of
 symmetric tie-point) in order 
to prove that this is not a consequence of Martin's Axiom (MA).

\begin{proposition}\cite{Vel93}
  It is consistent with MA and $\mathfrak c=\aleph_2$
  that there is an almost clopen set $A$ of $\mathbb N^*$ such that
   $x_A$ is a simple $P_{\aleph_2}$-point and,
  \begin{enumerate}
   \item  $x_A$ is a symmetric tie-point,
  \item there is an autohomeomorphism $f$ on $N^*$ such
    that $\{x\}$ is the only fixed point of $f$.
\end{enumerate}
\end{proposition}

Velickovic's result and approach was further generalized in
\cites{ShSt,Steprans}.
It is very interesting to know if an almost clopen subset of
$\mathbb N^*$ is itself homeomorphic to $\mathbb N^*$
(\cite{Farah, Just}).  This question also arose in the authors' work
on two-to-one images of $\mathbb N^*$ \cite{DSh2}.
Velickovic's method was slightly modified in \cite{DSh2} to produce a
complementary pair of almost clopen sets so that neither is
homeomorphic to $\mathbb N^*$, but it is not known 
if there is a symmetric tie-point $A\tie{x}B$ where $A$ is not a
copy of $\mathbb N^*$.

Our final mention of recent interest in almost clopen subsets of
$\mathbb N^*$ is in connection to the question \cite{LW,DR} of whether
the Banach space $\ell_\infty/c_0$ is necessarily primary. It was
noted
by Koszmider  
\cite{koszOpenII}*{p577} that a special almost clopen subset of
$\mathbb N^*$ could possibly resolve the problem. 
For a compact space $K$, we let  $C(K)$  denote the Banach space of
continuous real-valued functions on $K$ with the supremum norm. It is
well-known that $C(\mathbb N^*)$ is isomorphic (as a Banach space)
to $\ell_\infty/c_0$. Naturally if a space  $A$ 
 is homeomorphic to $\mathbb N^*$, then $C(A)$ is isomorphic to
$C(\mathbb N^*)$. 

\begin{proposition}\cite{koszOpenII}*{p577}
  Suppose that $A$ is an almost clopen subset of $ \mathbb N^*$
and that $B$ is its almost clopen complement.
If $C(\mathbb N^*)$ is not homeomorphic to either of
 $C(A)$ or $C(B)$, 
then $\ell_\infty/c_0$ is not primary.
\end{proposition}

\section{Asymmetric tie-points}

In many of the applications mentioned in the introductory section, the
tie-points utilized were symmetric tie-points. In other applications,
for example the primariness of  $\ell_\infty/c_0$,
it may be useful to find examples where the witnessing sets $A,B$ for
a tie-point are quite different.
There are any number of local topological properties that $x_A$ may
enjoy as a point in $A$ that it may not share as a point in $B$.
We make the following definition in analogy
 with simple $P_\kappa$-points.

 \begin{definition}
   Let $\kappa$ be a regular cardinal. An almost clopen set $A$
is simple of type $\kappa$ if $\mathcal I_A$ has a $\subset^*$-cofinal 
$\subset^*$-increasing chain $\{a_\alpha : \alpha\in \kappa\}$
 of type $\kappa$. 
\end{definition}

If $\{a_\alpha : \alpha \in \kappa\}$ is 
strictly $\subset^*$-increasing and
$\subset^*$-cofinal in $\mathcal I_A$ for an almost clopen set
 $A$, then the family $\{ a_{\alpha+1} \setminus a_\alpha : \alpha\in
 \kappa\}$ can not be \textit{reaped\/}. A family $\mathcal A
 \subset [\mathbb N]^{\aleph_0}$ is reaped by a set $c\subset\mathbb
 N$ if $|a\setminus c| = |a\cap c|$ for all $a\in \mathcal A$. The
 reaping number $\mathfrak r$ is the minimum cardinal of a family
 that can not be reaped \cite{GoSh}.  For any infinite
set $a\subset \mathbb
 N$,
 let $\next(a,\cdot)$ be the function in $\mathbb N^{\mathbb N}$
 defined by $\next(a,k) = \min(a\setminus\{1,\ldots,k\})$.
 As usual, for $f, g\in \mathbb N^{\mathbb N}$, we say
 that $f<^* g$ if $\{ k : g(k) \leq f(k)\}$ is finite. 

 \begin{proposition}\cite{GoSh}
   If\label{2.2}
 $\mathcal A\subset[\mathbb N]^{\aleph_0}$ and
   if there is some $g\in \mathbb N^{\mathbb N}$ such
   that      $ \next(a,\cdot) <^* g$ for all $a\in \mathcal A$,
    then $\mathcal A$ can be reaped. In particular, $\mathfrak b \leq
    \mathfrak r$.  
 \end{proposition}

Again, if $\{a_\alpha : \alpha \in \kappa\}$ is  strictly
 $\subset^*$-increasing and
$\subset^*$-cofinal in $\mathcal I_A$ for an almost clopen set
 $A$, then the family $\{ a_{\alpha+1} \setminus a_\alpha : \alpha\in
 \kappa\}$ is an example of a \textit{converging\/}
 family of infinite sets.

 \begin{definition}
   Let $\mathcal A$ be a family of infinite subsets of $\mathbb N$.
   We say that $\mathcal A$ converges if there is an ultrafilter $x$
   on $\mathbb N$ such that for each $U\in x$, 
    the set $\{ a\in \mathcal A : a\setminus U \neq^* \emptyset\}$ has
    cardinality less than that of $\mathcal A$.

    We say that $\mathcal A$ is hereditarily unreapable if each
    reapable subfamily of $\mathcal A$ has cardinality less than
that of     $\mathcal A$. 
 \end{definition}

 An ultrafilter $x$ of $\mathbb N^*$ is said to be an almost
$P_\kappa$-point if each set of fewer than $\kappa$ many members of
$x$ has a pseudointersection (an infinite set mod finite contained in
each of them).   Certainly a converging family is hereditarily
unreapable
and converges to a point that is an almost $P_\kappa$-point where
$\kappa$ is the cardinality of the family.
Clearly
 the cardinality of
 any hereditarily unreapable family will have cofinality
 less than 
 the splitting number $\mathfrak s$.
 First we recall that
a family $\mathcal A\subset [\mathbb N]^{\aleph_0}$ is a
\textit{splitting\/} family if for all infinite $b\subset \mathbb N$,
there is an $a\in \mathcal A$ such that $| b\cap a| = |b\setminus
a|$. We say that $b$ is split by $a$.
The splitting number, $\mathfrak s$, is the least cardinality of
a splitting family and $\mathfrak s\leq \mathfrak d$ (\cite{GoSh}).  
Therefore if, for example,
 $\mathfrak s
 =\aleph_1$ and $\mathfrak r =\mathfrak c = \aleph_2$,  there will
 be no hereditarily unreapable family. If $\mathfrak s = \mathfrak c$,
 then there is a hereditarily unreapable family of cardinality
 $\mathfrak s$. In the Mathias model, of $\mathfrak s = \mathfrak c
 =\mathfrak b=\aleph_2$,
 there is no converging unreapable family because there is no
  almost $P_{\aleph_2}$-point. In the Goldstern-Shelah model
  \cite{GoSh} of $\mathfrak r =\mathfrak s =  \aleph_1 < \mathfrak u$, 
there is
  (easily   checked) no converging family of cardinality $\mathfrak
  r$. It might be interesting to determine if there is
  a hereditarily unreapable family in this model because that would
imply there was  a stronger preservation result for the posets used.

If there is a simple almost clopen set of type $\kappa$, are there
restrictions on the behavior of its almost clopen complement and can
there be simple almost clopen sets of different types (including the
complement)? These are the types of questions that stimulated this
study. The most compelling of these has been answered.

\begin{theorem}
  If $A$ is\label{nodble} a simple almost clopen set of type $\kappa$ and
  if the complementary almost clopen set $B$ is simple, then it also
  has type $\kappa$. 
\end{theorem}

Similarly, there is a restriction on what the type of a simple
almost clopen set can be  that is shared by simple $P_\kappa$-points
(as shown by Nyikos (unpublished) see \cite{BlassSh}). 

\begin{theorem} If $A$ is a\label{oneof}
 simple almost clopen set of type $\kappa$,
  then $\kappa$ is one of $\{\mathfrak b, \mathfrak d\}$.
\end{theorem}

Now that we understand the limits on the behavior of a complementary
pair of simple  almost clopen sets, we look to the properties of the
complement $B$ when it is not assumed to be simple.  The topological
properties of character and tightness of $x_B$ in $B$ are natural
cardinal invariants to examine. These correspond to natural properties
of $\mathcal I_B$ as well. An indexed subset $\{ y_\beta :
\beta<\lambda\}$ of a space $X$ is said to be  a free
sequence if the closure of each initial segment is disjoint from the
closure of its complementary final segment. A $\lambda$-sequence
$\{y_\beta : \beta<\lambda\}$ is
converging if there is a point $y$ such that every neighborhood of $y$
contains a final segment of
$\{y_\beta : \beta<\lambda\}$. A subset $D$ of $\mathbb N^*$ is said
to be \textit{strongly\/} discrete \cite{FZ,Rabus2}
if there is a family of pairwise disjoint
clopen  subsets of $\mathbb N^*$  each containing a single point of
$D$. 

\begin{theorem}  If $\kappa<\lambda$ are uncountable
regular  cardinals
  with $\mathfrak c\leq \lambda$, 
 then\label{forcing}
 there is a ccc forcing extension in which
  there is a simple almost clopen set $A$ of type $\kappa$
  such that the almost clopen complement $B$ contains 
a  free $\lambda$-sequence $\{y_\beta : \beta < \lambda\}$
  that converges to  $x_A$.
\end{theorem}

We prove these theorems in the next section. We finish this section
by formulating some open problems about almost clopen sets and
possible asymmetries.

\begin{question}
  Can there exist simple almost clopen sets of different types?
\end{question}

\begin{question} Can there exist a simple almost clopen set of
  type greater than $\mathfrak b$?
\end{question}

\begin{question}
  If there is a simple almost clopen set of type $\kappa$
  is there a point of $\mathbb N^*$ of character $\kappa$?
  Is there a simple $P_\kappa$-point?
 \end{question}

The next question is simply a special case of the previous. 

 \begin{question}
   Is a simple almost clopen set of type $\aleph_1$ necessarily
   homeomorphic to $\mathbb N^*$?  
\end{question}

\begin{question}
   If $A$ is a simple almost clopen set of type $\kappa$, is
  there a simple almost clopen set $B'$ contained in the
  almost clopen complement $B$ of $A$ such that $x_A\in B'$?
  Is there a family of $\kappa$-many members of $\mathcal I_B$
  that converges to $x_A$?
\end{question}

Let us note that Theorem \ref{3.6} is pertinent to this
question.

\section{Proofs}

Our analysis of simple almost clopen sets depends on the connection
between the type of the clopen set and the ultrafilter ordering of
functions from $\mathbb N$ to $\mathbb N$. For an ultrafilter $x$ on
$\mathbb N$ the ordering $<_x$  is defined on $\mathbb N^{\mathbb N}$
by the condition that
$f<_x g$ if $\{ n\in \mathbb N : f(n) < g(n) \}\in x$. Since $x$ is an
ultrafilter, a set $F\subset \mathbb N^{\mathbb N}$ is cofinal in
$(\mathbb N^{\mathbb N}, <_x)$ if it is not bounded. Of course a
subset of $\mathbb N^{\mathbb N}$ that is unbounded with respect to
the $<_x$-ordering is also unbounded with respect to the mod finite
ordering $<^*$.

Fix a $<^*$-unbounded
 family $\{ f_\xi : \xi < \mathfrak b\}\subset \mathbb N^{\mathbb
   N}$ such that each $f_\xi$ is strictly increasing and such
 that $f_\eta<^* f_\xi$ for all $\eta < \xi <\mathfrak b$.
 The following well-known fact will be useful.

 \begin{proposition}
   For each infinite $b\subset \mathbb N$ and\label{nob}
   each unbounded
   $\Gamma\subset \mathfrak b$, 
 the family
     $\{ f_\xi \trestriction b : \xi \in \Gamma\}$ is $<^*$-unbounded
     in $\mathbb N^b$.
   \end{proposition}

   \begin{proof}
For each $\eta < \mathfrak b$, there is a $\xi\in \Gamma\setminus
\eta$ such that $f_\eta <^* f_\xi$, hence
$\{ f_\xi : \xi \in \Gamma\}$ is $<^*$-unbounded. If $g\in \mathbb
N^b$, then $g\circ \next(b,\cdot)\in \mathbb N^{\mathbb N}$.
So there is a $\xi\in \Gamma$ such that 
$f_\xi \not<^* g\circ \next(b,\cdot )$. 
Since $f_\xi$ is strictly increasing, $f_\xi\trestriction b\not<^* g$. 
\end{proof}

\begin{lemma} If a family $\mathcal A\subset [\mathbb N]^{\aleph_0}$
  converges to\label{isbdd}
 an ultrafilter $x$ and if $\{ f_\xi : \xi\in \mathfrak
  b\}$ is bounded mod $<_x$, then $\mathcal A$ has cardinality
  $\mathfrak b$.
\end{lemma}

\begin{proof}
Choose $g\in \mathbb N^{\mathbb N}$ so that $f_\xi <_x g$ for all
$\xi<\mathfrak b$. Since $\mathcal A$ can not be reaped,
 Proposition \ref{2.2} implies that $\mathfrak b\leq |\mathcal A|$.
For each $\xi$, let $U_\xi = \{ n\in \mathbb N : f_\xi(n)<g(n)\}\in
x$. If $\mathfrak b < |\mathcal A|$, then there is a $b\in \mathcal A$
such that $b\subset^* U_\xi$ for all $\xi< \mathfrak b$ (i.e. $x$ is
an almost $P_{\mathfrak b^+}$-point). However 
we would then have that $f_\xi\trestriction b <^* g\trestriction b$ 
for all $\xi<\mathfrak b$, and
by Proposition \ref{nob}, 
there is no such set $b$. This completes the proof.
\end{proof}

\begin{lemma}
 If a family $\mathcal A\subset [\mathbb N]^{\aleph_0}$
  converges\label{isunbdd}
 to an ultrafilter $x$ and if $\{ f_\xi : \xi\in \mathfrak
  b\}$ is unbounded mod $<_x$, then if
$\mathcal A$ has regular cardinality, that cardinal is
equal to  $\mathfrak d$.
\end{lemma}

\begin{proof}
Since we are assuming that $\{ f_\xi : \xi \in \mathfrak b\}$ is
$<_x$-unbounded, it is actually $<_x$-cofinal. 
We check that
the family $\{ f_\xi \circ \next(a,\cdot ) : \xi< \mathfrak b, a\in
\mathcal A\}$ is a $<^*$-dominating family. Take any strictly
increasing 
$g\in \mathbb
N^{\mathbb N}$ and choose $\xi<\mathfrak b$ such that
  $U= \{ n : g(n)<f_\xi(n) \}\in x$. Since $\mathcal A$ converges to
  $x$,
  there is an $a\in \mathcal A$ such that $a\subset^* U$.
Since $g$ is strictly increasing,
it is clear that $g < f_\xi\circ \next(U,\cdot )<^* f_\xi\circ
\next(a,\cdot)$.  Again, since $\mathcal A$ can not be reaped,
we have  $\mathfrak b\leq |\mathcal A|$ and  this implies
that
$\mathfrak d \leq |\mathcal A|$. Assume
that $\{ g_\beta : \beta < \mathfrak d\}\subset\mathbb N^{\mathbb N}$
 is a $<^*$-dominating family. For each $a\in
 \mathcal A$, there is a $\beta_a<\mathfrak d$ such that
 $\next(a,\cdot) <^* g_{\beta_a}$.  Now
since $\mathcal A$ is hereditarily  unreapable,
 Proposition \ref{2.2} implies that if $\mathcal A$ 
  has regular cardinality,
 the mapping $a\mapsto \beta_a$ is ${<}|\mathcal A|$-to-1.
This implies that $|\mathcal A|\leq \mathfrak d$.
\end{proof}

\begin{corollary}
Suppose that  $A$ is a simple\label{corol}
 almost clopen subset of $\mathbb N^*$ of type $\kappa$. 
    If  $\{ f_\xi : \xi<\mathfrak b\}$ is
    $<_{x_A}$-bounded, then
    $\kappa=\mathfrak b$; otherwise $\kappa = \mathfrak d$.
  \end{corollary}

  \begin{proof}
    Let $\{ a_\alpha : \alpha\in \kappa\}$ be the family contained in
    $\mathcal I_A$ witnessing that $A$ has type $\kappa$.
    Set
     $\mathcal A$ equal to
the family $ \{ a_{\alpha+1}\setminus a_\alpha : \alpha \in \kappa\}$
which converges to $x_A$. 
    If $\{ f_\xi : \xi<\mathfrak b\}$ is $<_{x_A}$-bounded,
    then by Lemma \ref{isbdd}, $\kappa = \mathfrak b$.
    Otherwise, since $\kappa$ is a regular cardinal, we have
    by Lemma \ref{isunbdd}, $\kappa = \mathfrak d$.    
\end{proof}
  
  \bgroup

  \def\proofname{Proof of Theorem \ref{nodble}.~}

  \begin{proof}
    Assume that $A$ and its
    complementary almost clopen set $B$ are both simple
    and let  $x=x_A$. 
If $\{ f_\xi : \xi< \mathfrak b\}$ is $<_x$-bounded
    then, by Corollary \ref{corol}
    they both have type $\mathfrak b$;  otherwise
    they both have type $\mathfrak d$.
  \end{proof}

   \def\proofname{Proof of Theorem \ref{oneof}.~}

   \begin{proof}
     Immediate  from Corollary
     \ref{corol}.
   \end{proof}

   \egroup

   We can improve Theorem \ref{nodble}. 

\begin{proposition} There is no\label{split}
 almost $P_{\mathfrak s^+}$-point in
  $\mathbb N^*$. 
\end{proposition}

\begin{proof}
  Let $\mathcal A$ be a splitting family of cardinality $\mathfrak s$.
  We may assume that $\mathcal A$ is closed under complements. 
  Let $x$ be any point of $\mathbb N^*$. It is easily seen
  that any pseudointersection of  $x\cap \mathcal A$
  is not split by any member of $\mathcal A$. Since $\mathcal A$
  is splitting, $x\cap \mathcal A$ has no pseudointersection,
  and so $x$ is not an almost $P_{\mathfrak s^+}$-point.
\end{proof}

Now we improve Theorem \ref{nodble}.

   \begin{theorem}
     If $A$ is a simple almost\label{3.6} 
 clopen set of type $\kappa$
      then $x_A$ is not an almost $P_{\kappa^+}$-point.
\end{theorem}

\begin{proof}
  We first note that by Proposition \ref{split} we
  must have that $\kappa <\mathfrak d$. Therefore, by Lemma
   \ref{isunbdd}, 
$\{ f_\xi : \xi <\mathfrak b\}$ is
$<_{x_A}$-bounded. Choose any $g\in \mathbb N^{\mathbb N}$ so that
$f_\xi <_{x_A} g$ for all $\xi < \mathfrak b$. For each
 $\xi$, let $U_\xi = \{ n\in \mathbb N : f_\xi(n) < g(n)\}$.
By Proposition
\ref{nob},
 we have that the collection $\{ U_\xi : \xi < \mathfrak b\}\subset x$
 has no pseudointersection.  By Theorem \ref{nodble},
$\mathfrak b\leq \kappa$ and  this proves the theorem.
\end{proof}

  Now we prove Theorem \ref{forcing}.  We first prove the easier
  special case when $\kappa = \aleph_1$. An
  $\alpha$-length finite support iteration
  sequence of posets, denoted $(\langle \mathbb P_\beta :\beta\leq
  \alpha\rangle , \langle \dot Q_\beta : \beta<\alpha\rangle)$,
  will mean that
  $\langle \mathbb P_\beta :\beta\leq \alpha\rangle
$ is an increasing chain of posets,
$\dot Q_\beta$ is a $\mathbb P_\beta$-name of a poset
   for each $\beta<\alpha$, 
 and
 members $p$ of $\mathbb P_\alpha$ will be functions with domain a
 finite subset, $\supp(p)$, of $\alpha$ satisfying that
 $p\trestriction \beta\in \mathbb P_\beta$ forces
 that $p(\beta)\in \dot Q_\beta$
  for $\beta\in \supp(p)$. As usual, $p_2<p_1$ providing
  $p_2\trestriction \beta \forces{\mathbb P_\beta}{p_2(\beta) <
    p_1(\beta)}$ for all $\beta\in \supp(p_1)$. Since $\mathbb P_0$ is
  the trivial poset,  we will allow ourselves to simply specify a
  poset $Q_0$ in such an iteration sequence rather than the
    $\mathbb P_0$-name of that poset.

\begin{definition}
Let $\mathcal A=\{ a_\beta : \beta < \alpha\}$ be a
$\subset^*$-increasing chain of subsets of $\omega$,
and\label{defineQ}
let $\mathcal I$ be an ideal contained in
$\mathcal A^\perp$. We define the poset
$Q = Q(\mathcal A;\mathcal I)$ where
$q\in Q$ if $q = (F_q, \sigma_q, b_q)$ where
\begin{enumerate}
\item $F_q\in [\omega]^{<\aleph_0}$, 
\item $b_q\in \mathcal I$ is disjoint from $F_q$,
  \item $\sigma_q : H_q \rightarrow \omega$ and $H_q\in [\alpha]^{<\aleph_0}$,
\item for each $\beta\in H_q$, $a_\beta\setminus \sigma_q(\beta)$ is
  disjoint from $b_q$.
\end{enumerate}
For $r,q\in Q$ we define $r<q$ providing $F_r\supset F_q$,
$\sigma_r \supset \sigma_q$, and
 $b_r\supset b_q$.
\end{definition}

  \begin{lemma}
  If $\mathcal A = \{a_\beta : \beta <\alpha\}$ is\label{noncf}
 a
  $\subset^*$-increasing chain of subsets of $\omega$ and if
  $\mathcal I$ is an ideal contained in $\mathcal A^\perp$,
  then $Q(\mathcal A; \mathcal I)$ is ccc whenever $\cf(\alpha)$ is
  not equal to $\omega_1$.  In addition,
  $Q(\mathcal A; [\omega]^{<\omega})$ is ccc for any
  infinite $\subset^*$-increasing chain $\mathcal A$.
\end{lemma}

  \begin{proof}
    The proofs are standard. The first statement 
is  basically the same
    as Theorem 4.2 of \cite{BaPFA}. 
For the last, $Q(\mathcal A; [\omega]^{<\omega})$ is ccc since
 conditions
 $p,q\in   Q(\mathcal A; [\omega]^{<\omega})$ are compatible
so long as  $F_p=F_q$, $b_p=b_q$, and $\sigma_p\cup \sigma_q$ is a
function. 
  \end{proof}

  \begin{definition}
If $Q$ is $Q(\mathcal A;\mathcal I)$ for some $\subset^*$-increasing
chain $\mathcal A$ of\label{Qgeneric}
subsets of $\omega$, and $\mathcal I\subset \mathcal A^\perp$
is an ideal, then the $Q$-generic set $\dot a_Q$ is defined as
the natural name  $\{ (\check F_q,  q) : q\in Q\}$, i.e. for each
$Q$-generic filter $G$, $\val_G(\dot a_Q)$ is equal to the union
of the family $\{ F_q : q\in G\}$. 
  \end{definition}

\begin{proposition}
For any $\subset^*$-increasing
chain $\mathcal A = \{a_\beta :\beta<\alpha\}$
 of\label{chain} subsets of $\omega$
and ideal  $\mathcal I\subset \mathcal A^\perp$,
$Q = Q(\mathcal A;\mathcal I)$ forces
that
$a_\beta\subset^* \dot a_Q$
and $\dot a_Q \cap  I=^*\emptyset$, 
for all $\beta\in \alpha$ and $I\in \mathcal I$.
\end{proposition}

\begin{proof}
  It is immediate that the set $D$ of $q=(F_q,\sigma_q,b_q)
\in Q(\mathcal A; \mathcal I)$
such that  $\beta\in \dom(\sigma_q)$ and $I\subset^* b_q$
is a dense set. Similarly, for each
$\ell\in \omega$,  $D_\ell = \{ q=(F_q,\sigma_q,b_q)\in Q
   : \ell \in F_q\cup b_q\}$ is also dense.
So it suffices to prove that
if $q\in D$, then
  $q\Vdash a_\beta\setminus\sigma_\beta \subset \dot a_Q$
  and that $q\Vdash \dot a_Q \cap I \subset (I\setminus b_q)$.
  Let $G$ be a generic filter with $q\in G$. If follows from the
  definition of the ordering on $Q$ that for $r<q$ in $G$,
  $I\setminus b_q \subset b_r$ and, since
  $F_r\cap b_r$ is empty, $F_r\cap I \subset I\setminus b_q$.
  Now let $\ell\in a_\beta \setminus \sigma_q(\beta)$ and choose
  $r\in G$ with $r<q$ and $r\in D_\ell$. Since
   $\sigma_r(\beta) = \sigma_q(\beta)$ and $a_\beta\setminus
   \sigma_r(\beta)$ is disjoint from $b_r$, it follows
   that $\ell\in F_r$. This means that $\ell \in \val_G(\dot a_Q)$.   
\end{proof}

    \begin{lemma} If $\lambda$ is a regular cardinal with $\mathfrak
    c\leq\lambda $, then there is a ccc forcing extension in which
    there is a simple almost clopen set $A$ of type $\omega_1$
    such\label{omega1case} 
 that there is a strongly discrete free $\lambda$-sequence
    converging to $x_A$.
  \end{lemma}

  \begin{proof}
There are ccc posets of cardinality 
 $\lambda$ that add a
    strictly $\subset^*$-increasing sequence $\{ b_\zeta :
    \zeta<\lambda\}$ of infinite subset of $\omega$ (e.g.
    \cite{KunenBook}*{II Ex. 22}).
    Alternatively, by Definition \ref{defineQ} and
     Lemma \ref{noncf},
we could let $Q_0$ be a $\lambda$-length
    finite support sequence of posets of the form
    $Q(\{b_\beta : \beta <\zeta\}; [\omega]^{<\omega})$
    and recursively let $b_\zeta$ be the resulting $\dot a_Q$ as
    in Definition \ref{Qgeneric}.
  
 For convenience we now work in
    such a ccc forcing extension and we construct a finite support
  ccc  iteration sequence of cardinality $\lambda$ and
length $\omega_1$ that will add
    a strictly $\subset^*$-increasing sequence
     $\{ a_\alpha : \alpha \in \omega_1\}$ of infinite subsets of
     $\omega$ so that the closure, $A$, of $\bigcup\{ a_\alpha^* :
     \alpha \in \omega_1\}$ is almost clopen. Suppose 
that we do this
in such a way that $\{b_\zeta : \zeta <\lambda\}$ is contained in
$\{ a_\alpha : \alpha\in \omega_1\}^\perp$ and, for all $U\in x_A$,
and all $\zeta<\lambda$, there is an $\eta<\lambda$ 
such that $U\cap (b_\eta\setminus b_\zeta)$ is infinite. 
 We check that there is then a strongly discrete free
 $\lambda$-sequence  converging to $x_A$.  Let $\{ U_\zeta : \zeta <
 \lambda\}$ enumerate  the members of $x_A$. 
 There is a cub $C\subset \lambda$ satisfying that
 for each $\delta\in C$,
 the family
  $\{ U_\xi : \xi<\delta\}$ is closed under finite intersections. 
Recursively define a
 strictly increasing 
function $g$ from $C$ into $\lambda$
 satisfying that $
U_\zeta\cap (b_{g(\delta)}\setminus b_\delta)$ is infinite
for all $\zeta<\delta\in C$.
Now, for each $\delta\in
 C$, let $x_\delta $ be an ultrafilter extending 
the family 
$\{ U_\zeta\cap (b_{g(\delta)}\setminus b_\delta) : \zeta<\delta\}$. 
Pass to a cub subset $C_1\subset C$ satisfying that
 $g(\eta)<\delta$ for all $\delta\in C$ and
 $\eta\in \delta\cap C$. It follows immediately
 that $\{ x_\delta : \delta \in C_1\}$ is strongly discrete
 and free. Similarly, the sequence converges to $x_A$ since
  $U_\zeta\in x_\delta$ for all $\zeta<\delta\in C_1$. 

  Now we  construct the iteration sequence to define the
$\subset^*$-increasing chain $\{a_\alpha : \alpha\in \omega_1\}$ that
will be cofinal in $\mathcal I_A$. We will use iterands of the form
$\dot Q_\alpha = 
Q(\{ \dot a_\beta : \beta < \alpha\}; \dot {\mathcal I}_\alpha)$ for
$0<\alpha<\omega_1$, and will recursively let
$\dot a_\alpha$ be the standard
$\mathbb P_{\alpha+1}$-name for $a_{\dot Q_\alpha}$ (as in
Definition \ref{Qgeneric}). The fact that 
$\{ a_\alpha : \alpha\in \omega_1\}$ will be a $\subset^*$-chain
follows from Proposition \ref{chain}.
Clearly the only choices we have
for the construction are the definition of $a_0$ and, by recursion,
the definition of $\dot {\mathcal I}_\alpha$. We will recursively
ensure that $\forces{\mathbb P_\alpha} {\{\dot a_\beta : \beta
  <\alpha\}
  \subset \{b_\zeta : \zeta<\lambda\}^\perp}$ simply by ensuring
that $\forces{\mathbb P_\alpha}{
  \{b_\zeta : \zeta<\lambda\} \subset \dot{\mathcal I}_\alpha}$.

To start the process, we let $\mathcal I_1$ be any maximal ideal
extending the ideal generated by $\{b_\zeta :\zeta<\lambda\}\cup
\{ b_\zeta : \zeta<\lambda\}^\perp$. Very likely
$\{b_\zeta : \zeta<\lambda\}^\perp$ is simply $[\omega]^{<\omega}$, so
 we 
 let $a_0$ be exceptional and equal  the emptyset. We now have our
 definition (working in the extension by $Q_0$)
 of $Q_1 = Q(\{ a_0\}; \mathcal I_1)$ and the generic set
  $\dot a_1 = \dot a_{Q_1}$ is forced to be almost disjoint from every
  member of $\mathcal I_1$ (it is a pseudointersection of the
ultrafilter      dual to $\mathcal I_1$). Now assume that
$\alpha<\omega_1$ and that we have defined $\dot{\mathcal I}_\beta$
for all $\beta <\alpha$. We recursively also ensure that,
 for $\beta < \gamma < \alpha$,
 the
 $\mathbb P_{\beta}$-name $\dot{\mathcal I}_{\beta+1}$ is a subset
 of the $\mathbb P_{\gamma}$-name $\dot {\mathcal I}_\gamma$,
 and that $\mathbb P_{\gamma}$ forces $\dot {\mathcal I}_\gamma$
 is contained in $\{\dot a_\beta : \beta <\gamma\}^\perp$.
 For the definition of $\dot{\mathcal I}_\alpha$ we break into
 three cases. 
If $\alpha$ is a limit ordinal, then we define
$\dot {\mathcal I}_\alpha$ to be the $\mathbb P_{\alpha}$-name of
the ideal $\{\dot a_\beta : \beta<\alpha \}^\perp$. By induction,
 we have, for $\gamma<\alpha$,
 that $\forces{\mathbb P_\alpha}{\dot{\mathcal I}_{\gamma+1}
   \subset \{\dot a_\beta : \beta<\alpha \}^\perp = \dot{\mathcal
     I}_\alpha}$, as required. In the case that $\alpha =\beta+1$
 for a successor $\beta$,
 we note that $\mathbb P_{\beta+1}$ forces that (by the
 genericity of $\dot a_\beta$)
the family
$\{ \dot a_\beta\} \cup \dot{\mathcal I}_\beta$
generates a proper ideal $\dot {\mathcal J}_\alpha$.
In the case that $\beta$ is a limit and $\alpha = \beta+1$,
 we note that $\mathbb P_{\beta+1}$ forces that 
the family
$\{ \dot a_\beta\} \cup \bigcup\{
\dot{\mathcal I}_{\gamma+1}:\gamma<\beta\} $ 
also generates a proper ideal $\dot {\mathcal J}_\alpha$.
Then, in either case where $\alpha = \beta+1$,
 we let $\dot{\mathcal J}'_\alpha$ be the
$\mathbb P_{\alpha}$-name of any maximal ideal that contains
$\dot{\mathcal J}_\alpha\cup (\dot{\mathcal J}_\alpha)^\perp$. The
definition of  $\dot{\mathcal I}_\alpha$ is then
the $\mathbb P_{\alpha}$-name of
$\dot{\mathcal  J}_\alpha' \cap \{ \dot a_\beta\}^\perp$.
For convenience, let $\dot y_{\beta+1}$ denote the
$\mathbb P_{\beta+1}$-name of this ultrafilter,
and let us notice that $\{ \omega\setminus (\dot a_\beta\cup b) : b\in
\dot{\mathcal I}_{\beta+1}\}$
is forced to be a base for
$\dot y_{\beta+1}$.
The set $\dot a_{\beta+1}\setminus \dot a_\beta$ will be a
pseudointersection
of $\dot y_{\beta+1}$.

 This completes the definition of the poset
   $\mathbb P_{\omega_1}$.
 Now we establish some properties.
Let $\dot A$ denote 
$\mathbb P_{\omega_1}$-name of the  closure in $\omega^*$
of the open set $\bigcup \{ \dot
 a_\alpha^* : \alpha\in \omega_1\}$. 
 
\bgroup

\def\proofname{Proof of Claim:~}

 \begin{claim}
   For each $\beta<\alpha <\omega_1$, $\mathbb P_{\alpha+1}$ forces
   that $\dot a_\alpha\setminus \dot a_{\beta}$ is a pseudointersection
   of    the filter $\dot y_{\beta+1}$. 
 \end{claim}

 \begin{proof}
We proceed by induction on $\alpha\geq\beta+1$. For $\alpha=\beta+1$,
   $\dot a_\alpha$ is almost disjoint
   from each member of $\dot{\mathcal I}_\alpha$, and so
   $\dot a_\alpha \setminus \dot a_\beta$ is almost
   disjoint from every member of $\dot{\mathcal J}_\alpha'$. Thus
   $\dot a_\alpha\setminus \dot a_\beta$ is forced to be mod finite
   contained in every member of the
   dual filter, namely
   $\dot y_{\beta+1}$.  Similarly, for $\alpha >\beta+1$,
   $\dot a_\alpha$ is forced to be almost disjoint from
   each member of $\dot {\mathcal I}_\alpha$. This means
   that $\dot a_\alpha$ is almost disjoint from each member
   of $\dot {\mathcal I}_{\beta+1}$, and so
   $\dot a_\alpha\setminus \dot a_{\beta+1}$ is also 
almost   disjoint from every member of $\dot{\mathcal J}_{\beta+1}'$.  
 \end{proof}

 \begin{claim}
   The 
family $ \{ \dot y_{\beta+1}  : \beta<\omega_1\}$ is a
family of $\mathbb P_{\omega_1}$-names and the union
is forced to 
generate
an ultrafilter
 $\dot x_{\dot A}$ that is indeed
 the unique boundary point of $\dot A$. 
 \end{claim}

 \begin{proof}
Since $\dot {\mathcal I}_{\beta+1} $ is contained in $\dot {\mathcal
  I}_{\alpha+1}$, and 
$\mathbb P_{\omega_1}$ forces that $\dot a_{\beta}\subset^* \dot a_\gamma$,
we have that $\mathbb P_{\omega_1}$ forces
 that $\bigcup \{ \dot y_{\beta+1} :
  \beta<\omega_1\}$ is a filter. 
  Furthermore, since $\mathbb P_{\omega_1}$ is ccc, every 
  $\mathbb P_{\omega_1}$-name of a subset of   $\omega$ is equal to
  a $\mathbb P_{\beta}$-name for some $\beta <\omega_1$. 
The fact that, for each $\beta<\omega_1$, 
   $\mathbb P_{\beta+1}$ forces that $\dot y_{\beta+1}$ is an
  ultrafilter implies that $\mathbb P_{\omega_1}$ forces that
    $\dot x_{\dot A}$ is an ultrafilter.
  Finally, it follows immediately from the previous claim
  that $\dot x_{\dot A}$ is the unique boundary point of $\dot A$.
 \end{proof}

 \begin{claim} For each $0<\alpha<\omega_1$, $\mathbb P_{\alpha+1}$
  forces
  that $\{ b_\zeta : \zeta <\lambda\}$ is
$\subset^*$-unbounded in
 $\dot{\mathcal I}_{\alpha+1}$.
\end{claim}

\begin{proof}
  We prove this by induction on $\alpha$.
We know that  $\mathbb P_{\alpha+1}$ forces that $\dot a_\alpha$ is
almost disjoint from every member of
$\{ b_\zeta : \zeta <\lambda\}$. Therefore,
 if  $\mathbb P_{\beta+1}$ forces that 
$\{ b_\zeta : \zeta <\lambda\}$ 
 is $\subset^*$-unbounded in
 $\dot{\mathcal I}_{\beta+1}$ for each $\beta <\alpha$, it follows that 
$\{ b_\zeta : \zeta <\lambda\}$  is $\subset^*$-unbounded 
 in what we called $\dot{\mathcal J}_{\alpha+1}$ above. 
In addition, we have that
 $\dot {\mathcal J}_{\alpha+1}^\perp \subset \{ b_\zeta : \zeta
 <\lambda\}^\perp$,  so 
 we have that $\mathbb P_{\alpha+1} $
 forces that $\{ b_\zeta : \zeta <\lambda\}$  is $\subset^*$-unbounded
 in $\dot {\mathcal I}_{\alpha+1}$. 
 \end{proof}

\egroup    

To finish the proof of the Lemma, we have to verify
that $\mathbb P_{\omega_1}$ forces that for all 
$U\in \dot x_{\dot A}$,
and all $\zeta<\lambda$, there is an $\eta<\lambda$ 
such that $U\cap (b_\eta\setminus b_\zeta)$ is infinite. 
Let $\zeta<\lambda$ be given and suppose that
$\dot U$ is a $\mathbb P_\alpha$-name of a member of
$\dot x_{\dot A}$ for some $\alpha<\omega_1$. Of course
this means that $\dot U$ is a member of $\dot y_{\alpha+1}$.
Now
consider the $\mathbb P_{\alpha+1}$-name,
 $\dot b$, of 
 the set $b_\zeta\cup \left(\omega\setminus (\dot U\cup  \dot
   a_{\alpha+1})\right)$.
 Evidently, $\dot b$ is forced to be disjoint from $\dot a_{\alpha+1}$
 and also is forced to not be in $\dot y_{\alpha+2}$. It follows
 that $\dot b$ is an element of $\dot {\mathcal I}_{\alpha+2}$.
 By Claim 3, there is an $\eta$ and a condition $p\in \mathbb
 P_{\alpha+2}$ such that $p$ forces that
 $b_\eta\setminus \dot b$ is infinite. Since $b_\eta\setminus
 \dot b$ is mod finite equal to $(b_\eta\setminus b_\zeta)\cap \dot
 U$, this completes the proof.
  \end{proof}

  To prove Theorem \ref{forcing}, we want to continue the recursive
  construction of Lemma \ref{omega1case} to a $\kappa$-length
  iteration of posets of the same form, namely
    $Q(\{\dot a_\beta : \beta <\alpha\} ; \dot {\mathcal
      I}_\alpha)$. It turns out that with the exact construction of
    Lemma
    \ref{omega1case}, $\mathbb P_{\omega_1}$ forces that
    $Q(\{\dot a_\beta : \beta<\omega_1\};
    \{ \dot a_\beta : \beta<\omega_1\}^\perp)$ is not ccc. For limit
    ordinals $\alpha$ of uncountable cofinality, it is likely
that we have
 to use $\{\dot a_\beta :\beta<\alpha\}^\perp$ as our choice
    for $\dot{\mathcal I}_\alpha$. However, we do have more
    flexibility at limits of countable cofinality and this
    is critical for extending the construction to any length $\kappa$.

\begin{definition}
We say that $\mathcal A$ is a pre-ccc sequence if
\begin{enumerate}
\item $\mathcal A = \{ a_\beta : \beta <\alpha\}$ for some
  increasing
   $\subset^*$-chain of subsets of $\omega$ with
   $\cf(\alpha)=\omega_1$,
   \item for each increasing  sequence $\{ \gamma_\xi : \xi\in
     \omega_1\}$ cofinal in $ \alpha$, and each sequence
     $\{ b_\xi : \xi\in \omega_1\}\subset \mathcal A^\perp$,
     such that $a_{\gamma_\xi}\cap b_\xi=\emptyset$ for all $\xi$,
     there are $\xi<\eta$ such that
      $a_{\gamma_\xi}\subset a_{\gamma_\eta}$ and $b_{\xi}\cap
      a_{\gamma_\eta}$ is empty.
\end{enumerate}
\end{definition}

\begin{lemma} If $\mathcal A$ is a pre-ccc sequence,
  then $Q(\mathcal A;\mathcal I)$ is ccc
  for any ideal $\mathcal I\subset \mathcal A^\perp$.
\end{lemma}

\begin{proof}
Let $\mathcal A = \{ a_\beta : \beta < \alpha\}$. 
  Let $\{ q_\xi : \xi \in \omega_1\}\subset Q = Q(\mathcal A;\mathcal
  I)$.    By passing to a subcollection we can suppose there is a single
  $F\in [\omega]^{<\aleph_0}$ such that $F_{q_\xi} = F$ for all
  $\xi$. For each $\xi$, let $b_\xi = b_{q_\xi}$, 
  $\sigma_\xi = \sigma_{q_\xi}$, and
  $H_\xi = \dom(\sigma_\xi)$.  By the standard 
$\Delta$-system lemma argument, we may assume that
$\sigma_\xi\cup\sigma_\eta$ is a function for all $\xi,\eta\in \omega_1$.

  For each $\xi$, let $\gamma_\xi$
  be the maximum element of $H_\xi$. By a trivial density argument we
  can assume that $\{\gamma_\xi : \xi\in\omega_1\}$ is a strictly
  increasing sequence that is cofinal in $\alpha$.

Next, we choose an integer $\bar m$ sufficiently large so that there
is again an countable $I\subset\omega_1$ and  a subset
$ \bar b$ of $\bar m$
such that, for all $\xi\in I$ and all $\beta\in H_\xi$
\begin{enumerate}
\item $\sigma_\xi (\beta) < \bar m$, 
\item $a_\beta\setminus \bar m\subset
  a_{\gamma_\xi}$,
\item $b_\xi\cap \bar m = \bar b$,
  \end{enumerate}
  Now we apply the pre-ccc property for the family
  $\{ \gamma_\xi : \xi\in I\}$ and the sequence
    $\{ b_\xi \setminus \bar m : \xi\in I\}$. Thus, we
may  choose $\xi<\eta$ from $I$ so that
  $a_{\gamma_\xi}\subset a_{\gamma_\eta}$ and
    $b_\xi\setminus \bar m$ is disjoint from $a_{\gamma_\eta}$. 
We claim that $r = (F, \sigma_\xi\cup \sigma_\eta, b_\xi\cup b_\eta)$
is in $Q$ and is an extension of each of $q_\xi$ and $q_\eta$.  It
suffices to  prove that for $\beta\in H_\xi$, $a_\beta\setminus
\sigma_\xi(\beta)$ is disjoint from $b_\eta$, and similarly, that
$a_\beta\setminus \sigma_\eta(\beta)$ is disjoint from $b_\xi$ for
 $\beta\in H_\eta$. Since $b_\xi\cap \bar m = b_\eta\cap \bar m = \bar
 b$, it suffices to consider $a_\beta\setminus \bar m$ in each case.
 For $\beta\in H_\xi$, we have $a_\beta\setminus \bar m \subset
  a_{\gamma_\xi}\subset a_{\gamma_\eta}$, and $a_{\gamma_\eta}$ is
  disjoint
  from $b_\eta\setminus \bar m$. For $\beta\in H_{\eta}$, we
  have $a_\beta\setminus \bar m \subset a_{\gamma_\eta}$ and
    $a_{\gamma_\eta}$ is disjoint from $b_\xi\setminus \bar m$.
  \end{proof}

  \begin{definition} $\mathfrak A$ is the class of triples
    $( \mathfrak P, \mathcal A, \mathfrak I)$ such
    that, there is an ordinal $\alpha$, and the
     following holds for each $\beta<\alpha$:
     \begin{enumerate}
     \item $\mathfrak  P
       =
(       \langle   \mathbb P_\beta :\beta\leq \alpha\rangle,
 \langle \dot Q_\beta  : \beta < \alpha   \rangle)$
 is a finite support iteration sequence of ccc posets,
     \item $\mathcal A $ is an
       $\alpha$-sequence $\{ \dot a_\beta : \beta < \alpha \}$,
       and \ $\forces{\mathbb P_{\beta+1}}{\dot a_\beta \subset \omega}$,
  \item $\mathfrak I$ is an $\alpha$-sequence
$ \{ \dot {\mathcal I}_\beta : \beta < \alpha \}$,
         \item $\forces{\mathbb
           P_\beta}{\dot{\mathcal I}_\beta
           \subset \{\dot a_\xi :
           \xi<\beta\}^\perp
           \ \mbox{is an ideal}}$
       \item for $\beta < \gamma < \alpha         $,
    $\forces {\mathbb P_{\gamma+1}}{\dot a_\beta\subset^* \dot
         a_\gamma} $
         and\  $\forces{\mathbb P_{\gamma}}{\dot{\mathcal
             I}_{\beta}\subset \dot{\mathcal I}_\gamma}$
         \item  $\forces{\mathbb P_\beta}{\dot
           Q_\beta = Q(\{\dot a_\xi : \xi<\beta\}; \dot {\mathcal
             I}_\beta)}$,
   \item 
 $\dot a_\beta$ is the $\mathbb P_{\beta+1}$-name for
           $\dot a_{Q(\{\dot a_\xi : \xi<\beta\}; \dot {\mathcal I}_\beta)}$
   \item if $\cf(\beta)=\omega$, then there is a sequence
     $\{ \dot{\mathcal I}_{\beta,\xi} : \xi<\beta\}$ such
     that $\dot{\mathcal I}_{\beta,\xi}$ is a $\mathbb P_{\xi}$-name
     and $\forces {\mathbb P_\beta}{\dot{\mathcal I}_{\beta}
       =\bigcup\{ \dot{\mathcal I}_{\beta,\xi} : \xi<\beta\}}$.
     \end{enumerate}
  \end{definition}

  \begin{lemma} If $\alpha$ is an ordinal with cofinality $\omega_1$,
    and\label{willbepreccc}
 if
    $( 
       \langle   \mathbb P_\beta :\beta\leq \alpha\rangle,
 \langle \dot Q_\beta  : \beta < \alpha   \rangle)
, \{ \dot a_\beta : \beta < \alpha \},
     \{ \dot {\mathcal I}_\beta : \beta < \alpha \})$ is in $\mathfrak
     A     $ then
$\forces{\mathbb P_\alpha}{ \{\dot a_\beta : \beta < \alpha\}
 \ \mbox{is a pre-ccc sequence}}$.
  \end{lemma}

  \begin{proof}
Let $\dot{\mathcal A}$ denote the $\mathbb P_\alpha$-name of the 
 sequence $\{ \dot a_\beta : \beta <\alpha\}$. 
Let $\{ \dot{\gamma}_\xi : \xi\in \omega_1\}$ and
  $\{\dot b_\xi : \xi\in \omega_1\}$ be sequences of 
 $\mathbb P_{\alpha}$-names such that there is some $p_0\in \mathbb
 P_\alpha$ forcing that, for each $\xi<\omega_1$, 
  $\dot\gamma_\xi \in \alpha$, 
  $\dot b_\xi\in \dot {\mathcal A}^\perp$ and $
 \dot b_\xi \cap \dot a_{\gamma_\xi}$ is empty. Suppose
also that $p_0$ forces that $\{ \dot \gamma_\xi : \xi \in \omega_1\}$
is strictly increasing and cofinal in $\alpha$. We may assume that
$p_0$ decides the value, $\gamma_0$,  of $\dot \gamma_0$.
 For each
 $\xi<\omega_1$, choose any $p_\xi < p_0$ that decides a value,
 $\gamma_\xi$, of $\dot \gamma_\xi$
 and that $\dot b_\xi$ is a $\mathbb P_\beta$-name for some
 $\beta\in \supp(p_\xi)$. 

 Let $g$ be a continuous
strictly increasing function from $\omega_1$ into $\alpha$ with
cofinal range. Since $\mathbb P_\alpha$ is ccc we have, 
for each $\delta\in \omega_1$, 
the set $\{ \xi : \gamma_\xi < g(\delta)\}$ is countable. Therefore
there is a cub $C\subset\omega_1$ such that $g(\delta)\leq
\gamma_\delta$ for all $\delta\in C$. We may also arrange
that, for each $\delta\in C$ and $\xi<\delta$,
$\supp(p_\xi)\subset g(\delta)$.

For each $\delta\in C$, we may extend $p_\delta$ so as to ensure
that each of 
$g(\delta)$ and $\gamma_\delta$
are in  $\supp(p_\delta)$,
and such that there is a $\beta\in \supp(p_\xi)$ such
 that $\dot b_\xi$ is a $\mathbb P_{\beta}$-name.
We also extend each $p_\delta$ so that we can arrange a list of
special properties (referred  to as ``determined'' in many similar
constructions). Specifically, for each $\beta\in \supp(p_\delta)$,
\begin{enumerate}
\item there are $F^\delta_\beta\in [\omega]^{<\aleph_0}$,
  $H^\delta_\beta \in [\beta]^{<\aleph_0}$, $\sigma^\delta_\beta :
   H^\delta_\beta \rightarrow \omega$, and a $\mathbb P_\beta$-name
   $\dot b^\delta_\beta$ such that
    $p_\delta\trestriction \beta \forces{\mathbb
      P_\beta}{p_\delta(\beta)
      = (F^\delta_\beta, \sigma^\delta_\beta, \dot b^\delta_\beta)}$,
\item if $g(\delta)<\gamma_\delta$, then $g(\delta)\in
  H^\delta_{\gamma_\delta}$, 
\item $H^\delta_\beta\subset \supp(p_\delta)$,
    \item if $\beta$ is a limit with countable cofinality, then
      there is a $\mu^\delta_\beta<\beta$ such that
        $\dot b^\delta_\beta$ is a $\mathbb
        P_{\mu^\delta_\beta}$-name and $\supp(p_\delta)\cap
        \beta\subset
         \mu^\delta_\beta$,
      \item if $\iota <\beta$ is in $\supp(p_\delta)$, then
                $H^\delta_\iota\supset H^\delta_\beta\cap\iota$,
              \end{enumerate}

By the pressing down lemma,  there is a stationary set               
$S\subset C$
and a $\mu<\alpha$ such that $\mu^\delta_{g(\delta)}<\mu$
for all $\delta\in S$. Now let $G_{\mu+1}$ be $\mathbb
P_{\mu+1}$-generic filter satisfying that, in the extension, 
there is a stationary set $S_1\subset S$ so that
$p_\xi\trestriction\mu\in G_{\mu+1}$ for all $\xi\in S_1$.
We may also arrange that the values of the pair $\{F^\xi_{g(\xi)},
F^\xi_{\gamma_\xi}\}$ is the same for all $\xi\in S_1$.
For all $\beta\in \mu+1$, we let $a_\beta$ denote the valuation of
 $\dot a_\beta$ by $G_{\mu+1}$.
By further
shrinking $S_1$ we may suppose there is an $\bar m\in\omega$
and a $\bar b\subset\bar m$,
 satisfying that, for all $\delta\in S_1$,
 \begin{enumerate}
 \item  for all $\beta\in \supp(p_\delta)$
   $F^\delta_\beta\subset \bar m$, 
and, for  all $\iota\in H^\delta_\beta$,
   $\sigma^\delta_\beta(\iota) < \bar   m$,
  \item for all $\beta\in \supp(p_\delta)\cap \mu$, $a_\iota\setminus
    a_\mu \subset \bar m$
    \item $\bar b = \bar m\cap b^\delta_{g(\delta)}$, where
$b^\delta_{g(\delta)}$ is  the valuation of $\dot
b^\delta_{g(\delta)}$ by $G_{\mu+1}$,
\item $b^\delta_{g(\delta)}\cap a_\mu \subset \bar m$.
  \end{enumerate}

\noindent  Fix any $\xi<\eta$ from $S_1$. Define $q_\xi$ so that
$\supp(q_\xi) =\supp(p_\xi)\setminus \mu+1$,
and, for $\beta\in \supp(q_\xi)$,
\[q_\xi(\beta) = \begin{cases}
(F^\xi_{g(\xi)},
      \sigma^\xi_{g(\xi)}\cup \{(\mu,\bar m)\}, \dot
      b^\xi_{g(\xi)}\cup
      \dot b^\eta_{g(\eta)}) & \mbox{if}\ \beta = g(\xi)\\
 (F^\xi_\beta, \sigma^\xi_\beta, \dot b^\xi_\beta\cup
      b^\eta_{g(\eta)}\setminus\bar m)& \mbox{if}\ g(\xi)<\beta\ . 
    \end{cases}
  \]

    \noindent We prove by induction on $\beta\in \supp(q_\xi)$,
    that there is a condition  $r^\xi_\beta\in G_{\mu+1}$ 
such    that $r^\xi_\beta\cup (q_\xi\trestriction (\beta+1)) \leq
p_\xi\trestriction (\beta+1)$.
 Evidently, for the case $\beta=g(\xi)$, 
 $F^\xi_{g(\xi)}$ and $\bar m\cap a_\iota$ are disjoint
from $\bar b$ and so 
 there is some condition in $G_{\mu+1}$ that forces that,
they are 
disjoint from $ \dot b^\eta_{g(\eta)}$. Similarly, 
for $\iota\in H^\xi_{g(\xi)}$, $a_\iota\setminus \bar m\subset a_\mu$,
and since $a_\mu\setminus \bar m \cap (b^\xi_{g(\xi)}\cup
b^\eta_{g(\eta)})$ is empty, there is a condition $r$
in $ G_{\mu+1}$ that
forces that $q_\xi(g(\xi))\in \dot Q_{g(\xi)}$
and that   $q_\xi(g(\xi)) < p_\xi(g(\xi))$.
In addition, $r\cup q_\xi\trestriction (g(\xi)+1)$ forces that
 $\dot a_{g(\xi)}$ is disjoint from $b^\eta_{g(\eta)}\setminus \bar
 m$.  Now, suppose that 
$g(\xi)<\beta\in \supp(p_\xi)$, and
 that $r\cup q\trestriction \beta $ is a condition
 in $\mathbb P_\beta$ that is below $p_\xi\trestriction \beta$. We
 recall that $H^\xi_\beta\subset \supp(p_\xi)$, and so it follows
 that $r\cup q_\xi\trestriction \beta$ forces that $\dot a_\iota $
 is disjoint from $b^\eta_{g(\eta)}\setminus \bar m$ for all $\iota\in
 H^\xi_\beta$. This is the only thing that needs verifying when
 checking that $r\cup q_\xi\trestriction (\beta+1) <
  p_\xi\trestriction (\beta+1)$. 

  Now that we have that $r\cup q_\xi$ forces that
  $\dot a_{\gamma_\xi}$ is disjoint from $b^\eta_{g(\eta)}\setminus
  \bar m$, we can add $\{(\gamma_\xi,\bar m)\}$ to
  $\sigma^\eta_{g(\eta)}$ and still have a condition. Similarly,
  for all $\iota\in \supp(p_\eta)\cap \mu$, $a_\iota\setminus \bar m$
  is contained in $a_\mu$,
  and    $r\cup q_\xi$ forces that $a_\mu\setminus \bar m$,
   being a subset of $
 \dot a_{\gamma_\xi}$, is disjoint from $\dot b_\xi$. 
This implies that $r\cup q_\xi$ forces that 
  $(F^\eta_{g(\eta)}, \sigma^\eta_{g(\eta)}\cup \{(\gamma_\xi,\bar m)\},
b^\eta_{g(\eta)}\cup (\dot b_\xi\setminus \bar m))$ is a condition
in $\dot Q_{g(\eta)}$ and is less than $p_\eta(g(\eta))$. 
Now we define a condition $q_\eta$ so that $\supp(q_\eta) =
\supp(p_\eta)\setminus \mu+1$, and, for $\beta\in \supp(q_\eta)$,
\[q_\eta(\beta) = \begin{cases}
(F^\eta_{g(\eta)},
      \sigma^\eta_{g(\eta)}\cup \{(\gamma_\xi,\bar m)\}, \dot
      b^\eta_{g(\eta)}\cup
      \dot b_{\gamma_\xi}) & \mbox{if}\ \beta = g(\xi)\\
 (F^\eta_\beta, \sigma^\eta_\beta, \dot b^\eta_\beta\cup
     (\dot  b_{\gamma_\xi}\setminus\bar m))& \mbox{if}\ 
g(\eta)<\beta\ . 
    \end{cases}
  \]
\noindent It again follows, by induction on $\beta\in \supp(q_\eta)$,
that $r\cup q_\xi$ forces that $r\cup q_\xi\cup (q_\eta\trestriction
(\beta+1))$ is a condition in $\mathbb P_{\beta+1}$ and is below
$p_\eta\trestriction (\beta+1)$. Finally, we observe that
$r\cup q_\xi\cup q_\eta$ forces that $\dot a_{\gamma_\xi}\subset
\dot a_{\gamma_\eta}$ because it forces that
$\dot a_{\gamma_\xi} \cap \bar m =  \dot a_{\gamma_\eta} \cap \bar m$
and that
$\dot a_{\gamma_\xi}\setminus  \bar m \subset \dot
a_{g(\eta)}\setminus
\bar m\subset \dot a_{\gamma_\eta}$. Similarly
$r\cup q_\xi\cup q_\eta$ forces that $\dot a_{\gamma_\eta}$
is disjoint from $\dot b_\xi$ because
$\dot b_\xi\cap \bar m = \bar b$ and
$\dot a_{\gamma_\eta}$ is disjoint from  $\dot b_\xi\setminus \bar m $.
\end{proof}

  \begin{corollary}
If\label{willbeccc}     
$( 
       \langle   \mathbb P_\beta :\beta\leq \alpha\rangle,
 \langle \dot Q_\beta  : \beta < \alpha   \rangle)
, \{ \dot a_\beta : \beta < \alpha \},
     \{ \dot {\mathcal I}_\beta : \beta < \alpha \})$ is in $\mathfrak
     A     $, then 
$\forces {\mathbb P_\alpha}{
       Q(\{\dot a_\beta : \beta < \alpha\}; \dot{\mathcal I}) \ \mbox{is ccc}}
     $ for each $\mathbb P_\alpha$-name $\dot{\mathcal I}$ satisfying
     that $\forces{\mathbb P_\alpha}{\dot{\mathcal I} \subset
       \{\dot a_\beta : \beta <\alpha \}^\perp
       \ \mbox{is an ideal}}$.      
   \end{corollary}

   \bgroup

\def\proofname{Proof of Theorem \ref{forcing}\/.~}

\begin{proof}
  We simply adapt the proof of Lemma \ref{omega1case} so as to ensure
  that, for each $\alpha<\kappa$,
 $$( 
       \langle   \mathbb P_\beta :\beta\leq \alpha\rangle,
 \langle \dot Q_\beta  : \beta < \alpha   \rangle)
, \{ \dot a_\beta : \beta < \alpha \},
\{ \dot {\mathcal I}_\beta : \beta < \alpha \})\
\mbox{is in }\ \mathfrak
A     ~~.$$
The only change to the proof is that when $\beta<\kappa$ is a limit
with countable cofinality, the definition of $\dot{\mathcal I}_\beta$
is equal to the union of the sequence
$\{ \dot{\mathcal I}_{\xi+1} : \xi < \beta\}$. With this change,
 we recursively have ensured that our iterations remain in $\mathfrak
 A$, and by Corollary \ref{willbeccc}, our iteration is ccc.
 All the details showing that the closure,
 $\dot A$, of the union of the chain $\{ \dot a_\alpha^* :
 \alpha<\kappa\}$ is forced to be almost clopen go through as
 in Lemma \ref{omega1case}. Similarly, it follows by recursion
 that the initial family $\{ b_\zeta :\zeta<\lambda\}$ is a
 $\subset^*$-unbounded subfamily of the ideal
    $\{ \dot a_\alpha : \alpha <\kappa\}^\perp$.
\end{proof}

   \egroup


\begin{bibdiv}

\def\cprime{$'$} 

\begin{biblist}

\bib{BaPFA}{article}{
   author={Baumgartner, James E.},
   title={Applications of the proper forcing axiom},
   conference={
      title={Handbook of set-theoretic topology},
   },
   book={
      publisher={North-Holland, Amsterdam},
   },
   date={1984},
   pages={913--959},
   review={\MR{776640}},
}	

\bib{BlassSh}{article}{
   author={Blass, Andreas},
   author={Shelah, Saharon},
   title={There may be simple $P_{\aleph_1}$- and $P_{\aleph_2}$-points and
   the Rudin-Keisler ordering may be downward directed},
   journal={Ann. Pure Appl. Logic},
   volume={33},
   date={1987},
   number={3},
   pages={213--243},
   issn={0168-0072},
   review={\MR{879489}},
   doi={10.1016/0168-0072(87)90082-0},
}

\bib{vDHandbook}{collection}{
   title={Handbook of set-theoretic topology},
   editor={Kunen, Kenneth},
   editor={Vaughan, Jerry E.},
   publisher={North-Holland Publishing Co., Amsterdam},
   date={1984},
   pages={vii+1273},
   isbn={0-444-86580-2},
   review={\MR{776619 (85k:54001)}},
}


\bib{vDKvM}{article}{
   author={van Douwen, Eric K.},
   author={Kunen, Kenneth},
   author={van Mill, Jan},
   title={There can be $C^*$-embedded dense proper subspaces in
   $\beta\omega-\omega$},
   journal={Proc. Amer. Math. Soc.},
   volume={105},
   date={1989},
   number={2},
   pages={462--470},
   issn={0002-9939},
   review={\MR{977925}},
   doi={10.2307/2046965},
}
 
\bib{DSh1}{article}{
   author={Dow, Alan},
   author={Shelah, Saharon},
   title={Tie-points and fixed-points in $\mathbb N^*$},
   journal={Topology Appl.},
   volume={155},
   date={2008},
   number={15},
   pages={1661--1671},
   issn={0166-8641},
   review={\MR{2437015}},
   doi={10.1016/j.topol.2008.05.002},
}

\bib{DSh2}{article}{
   author={Dow, Alan},
   author={Shelah, Saharon},
   title={More on tie-points and homeomorphism in $\mathbb N^\ast$},
   journal={Fund. Math.},
   volume={203},
   date={2009},
   number={3},
   pages={191--210},
   issn={0016-2736},
   review={\MR{2506596}},
   doi={10.4064/fm203-3-1},
}

\bib{DSh3}{article}{
   author={Dow, Alan},
   author={Shelah, Saharon},
   title={An Efimov space from Martin's axiom},
   journal={Houston J. Math.},
   volume={39},
   date={2013},
   number={4},
   pages={1423--1435},
   issn={0362-1588},
   review={\MR{3164725}},
}

\bib{DR}{article}{
   author={Drewnowski, Lech},
   author={Roberts, James W.},
   title={On the primariness of the Banach space $l_{\infty}/C_0$},
   journal={Proc. Amer. Math. Soc.},
   volume={112},
   date={1991},
   number={4},
   pages={949--957},
   issn={0002-9939},
   review={\MR{1004417}},
   doi={10.2307/2048638},
}
 


\bib{Farah}{article}{
   author={Farah, Ilijas},
   title={Analytic quotients: theory of liftings for quotients over analytic
   ideals on the integers},
   journal={Mem. Amer. Math. Soc.},
   volume={148},
   date={2000},
   number={702},
   pages={xvi+177},
   issn={0065-9266},
   review={\MR{1711328}},
   doi={10.1090/memo/0702},
}


\bib{FZ}{article}{
   author={Frankiewicz, R.},
   author={Zbierski, P.},
   title={Strongly discrete subsets in $\omega^*$},
   journal={Fund. Math.},
   volume={129},
   date={1988},
   number={3},
   pages={173--180},
   issn={0016-2736},
   review={\MR{962539}},
}


\bib{FG}{article}{
   author={Fine, N. J.},
   author={Gillman, L.},
   title={Extension of continuous functions in $\beta N$},
   journal={Bull. Amer. Math. Soc.},
   volume={66},
   date={1960},
   pages={376--381},
   issn={0002-9904},
   review={\MR{0123291}},
   doi={10.1090/S0002-9904-1960-10460-0},
}

\bib{GoSh}{article}{
   author={Goldstern, M.},
   author={Shelah, S.},
   title={Ramsey ultrafilters and the reaping number---${\rm Con}({\germ
   r}<{\germ u})$},
   journal={Ann. Pure Appl. Logic},
   volume={49},
   date={1990},
   number={2},
   pages={121--142},
   issn={0168-0072},
   review={\MR{1077075}},
   doi={10.1016/0168-0072(90)90063-8},
}

\bib{JKS}{article}{
   author={Juh\'asz, Istv\'an},
   author={Koszmider, Piotr},
   author={Soukup, Lajos},
   title={A first countable, initially $\omega_1$-compact but non-compact
   space},
   journal={Topology Appl.},
   volume={156},
   date={2009},
   number={10},
   pages={1863--1879},
   issn={0166-8641},
   review={\MR{2519221}},
   doi={10.1016/j.topol.2009.04.004},
}

\bib{Just}{article}{
   author={Just, Winfried},
   title={Nowhere dense $P$-subsets of $\omega$},
   journal={Proc. Amer. Math. Soc.},
   volume={106},
   date={1989},
   number={4},
   pages={1145--1146},
   issn={0002-9939},
   review={\MR{976360}},
   doi={10.2307/2047305},
}

\bib{Katetov}{article}{
   author={Kat\v etov, M.},
   title={A theorem on mappings},
   journal={Comment. Math. Univ. Carolinae},
   volume={8},
   date={1967},
   pages={431--433},
   issn={0010-2628},
   review={\MR{0229228}},
}

\bib{Kopp}{article}{
   author={Koppelberg, Sabine},
   title={Minimally generated Boolean algebras},
   journal={Order},
   volume={5},
   date={1989},
   number={4},
   pages={393--406},
   issn={0167-8094},
   review={\MR{1010388}},
   doi={10.1007/BF00353658},
}

\bib{KoszTAMS}{article}{
   author={Koszmider, Piotr},
   title={Forcing minimal extensions of Boolean algebras},
   journal={Trans. Amer. Math. Soc.},
   volume={351},
   date={1999},
   number={8},
   pages={3073--3117},
   issn={0002-9947},
   review={\MR{1467471}},
   doi={10.1090/S0002-9947-99-02145-5},
}

\bib{KunenBook}{book}{
   author={Kunen, Kenneth},
   title={Set theory},
   series={Studies in Logic and the Foundations of Mathematics},
   volume={102},
   note={An introduction to independence proofs},
   publisher={North-Holland Publishing Co., Amsterdam-New York},
   date={1980},
   pages={xvi+313},
   isbn={0-444-85401-0},
   review={\MR{597342}},
}


\bib{LW}{article}{
   author={Leonard, I. E.},
   author={Whitfield, J. H. M.},
   title={A classical Banach space: $l_{\infty }/c_{0}$},
   journal={Rocky Mountain J. Math.},
   volume={13},
   date={1983},
   number={3},
   pages={531--539},
   issn={0035-7596},
   review={\MR{715776}},
   doi={10.1216/RMJ-1983-13-3-531},
}

\bib{koszOpenII}{collection}{
   title={Open problems in topology. II},
   editor={Pearl, Elliott},
   publisher={Elsevier B. V., Amsterdam},
   date={2007},
   pages={xii+763},
   isbn={978-0-444-52208-5},
   isbn={0-444-52208-5},
   review={\MR{2367385}},
}

\bib{Rabus2}{article}{
   author={Rabus, Mariusz},
   title={On strongly discrete subsets of $\omega^\ast$},
   journal={Proc. Amer. Math. Soc.},
   volume={118},
   date={1993},
   number={4},
   pages={1291--1300},
   issn={0002-9939},
   review={\MR{1181172}},
   doi={10.2307/2160090},
}

\bib{Rabus}{article}{
   author={Rabus, Mariusz},
   title={An $\omega_2$-minimal Boolean algebra},
   journal={Trans. Amer. Math. Soc.},
   volume={348},
   date={1996},
   number={8},
   pages={3235--3244},
   issn={0002-9947},
   review={\MR{1357881}},
   doi={10.1090/S0002-9947-96-01663-7},
}


\bib{Sapirbpoint}{article}{
   author={\v Sapirovski\u\i , B. \`E.},
   title={The imbedding of extremally disconnected spaces in bicompacta.
   $b$-points and weight of pointwise normal spaces},
   language={Russian},
   journal={Dokl. Akad. Nauk SSSR},
   volume={223},
   date={1975},
   number={5},
   pages={1083--1086},
   issn={0002-3264},
   review={\MR{0394609}},
}	

\bib{ShStPFA}{article}{
   author={Shelah, Saharon},
   author={Stepr\=ans, Juris},
   title={PFA implies all automorphisms are trivial},
   journal={Proc. Amer. Math. Soc.},
   volume={104},
   date={1988},
   number={4},
   pages={1220--1225},
   issn={0002-9939},
   review={\MR{935111}},
   doi={10.2307/2047617},
}

\bib{ShSt}{article}{
   author={Shelah, Saharon},
   author={Stepr\=ans, Juris},
   title={Somewhere trivial autohomeomorphisms},
   journal={J. London Math. Soc. (2)},
   volume={49},
   date={1994},
   number={3},
   pages={569--580},
   issn={0024-6107},
   review={\MR{1271551}},
   doi={10.1112/jlms/49.3.569},
}



\bib{Steprans}{article}{
   author={Stepr\=ans, Juris},
   title={The autohomeomorphism group of the \v Cech-Stone compactification of
   the integers},
   journal={Trans. Amer. Math. Soc.},
   volume={355},
   date={2003},
   number={10},
   pages={4223--4240},
   issn={0002-9947},
   review={\MR{1990584}},
   doi={10.1090/S0002-9947-03-03329-4},
}


\bib{Vel93}{article}{
   author={Veli\v ckovi\'c, Boban},
   title={${\rm OCA}$ and automorphisms of ${\scr P}(\omega)/{\rm fin}$},
   journal={Topology Appl.},
   volume={49},
   date={1993},
   number={1},
   pages={1--13},
   issn={0166-8641},
   review={\MR{1202874}},
   doi={10.1016/0166-8641(93)90127-Y},
}

		

\end{biblist}
\end{bibdiv}
\end{document}